%% file: Update.tex
\documentclass[11pt]{amsart}

\usepackage{amsfonts, amstext, amsmath, amsthm, amscd, amssymb}
\usepackage{epsfig, graphics, psfrag}
\usepackage{color}
\usepackage{url}
\usepackage[margin=1.2in]{geometry}
\usepackage{adjustbox}

\let\oldmarginpar\marginpar
\renewcommand\marginpar[1]{\oldmarginpar[\raggedleft\footnotesize #1]%
{\raggedright\footnotesize #1}}

\renewcommand{\setminus}{{\smallsetminus}}

\newcommand{\RR}{{\mathbb{R}}}
\newcommand{\Z}{{\mathbb{Z}}}

\newcommand{\C}{{\mathbb{C}}}
\newcommand{\QQ}{{\mathbb{Q}}}

\newcommand{\vol}{{\rm vol}}

\newcommand{\Tr}{\mathrm{Tr}}

\theoremstyle{plain}
\newtheorem{theorem}{Theorem}[section]
\newtheorem{corollary}[theorem]{Corollary}
\newtheorem{lemma}[theorem]{Lemma}
\newtheorem{proposition}[theorem]{Proposition}

\newtheorem{conjecture}[theorem]{Conjecture}

\newtheorem*{namedtheorem}{\theoremname}
\newcommand{\theoremname}{testing}
\newenvironment{named}[1]{\renewcommand{\theoremname}{#1}\begin{namedtheorem}}{\end{namedtheorem}}

\theoremstyle{definition}
\newtheorem{definition}[theorem]{Definition}
\newtheorem{remark}[theorem]{Remark}

\title{Quantum representations and monodromies of fibered links}

\author{Renaud Detcherry}
\author{Efstratia Kalfagianni}

\address[]{Department of Mathematics, Michigan State University, East
Lansing, MI, 48824, USA}

\email[]{kalfagia@math.msu.edu}

\email[]{renaud.detcherry@gmail.com}

\begin{document}

\thanks{Research supported by NSF Grants DMS-1404754 and DMS-1708249}
\date{\today}

\begin{abstract}  Andersen, Masbaum and Ueno conjectured that 
certain quantum representations of surface mapping class groups should send pseudo-Anosov mapping classes to elements of infinite order (for large enough level $r$). 
In this paper, we relate the AMU conjecture to a question about the growth of the Turaev-Viro invariants $TV_r$ of hyperbolic 3-manifolds. We show that if the $r$-growth of  $|TV_r(M)|$ for a hyperbolic  3-manifold $M$ that fibers over the circle is exponential,  then the monodromy of the fibration  of $M$ satisfies the AMU conjecture.  
Building on earlier work \cite{DK} we give broad constructions of (oriented) hyperbolic fibered links, of arbitrarily high genus,  whose  $SO(3)$-Turaev-Viro invariants have exponential $r$-growth.
As a result, for any $g>n\geqslant 2$, we
 obtain  infinite families of non-conjugate pseudo-Anosov mapping classes, acting on surfaces of genus $g$ and $n$ boundary components, that satisfy the AMU conjecture.
 
 We also discuss integrality properties of the traces of quantum representations and we answer a question of Chen and Yang about Turaev-Viro invariants of torus links. 
\end{abstract}

\vskip 0.1in



\maketitle

\section{Introduction}
Given a compact oriented surface $\Sigma,$ possibly with boundary, the mapping class group $\mathrm{Mod}(\Sigma)$ is the group of isotopy classes
of orientation preserving homeomorphisms of $\Sigma$ that fix the boundary.
The  Witten-Reshetikhin-Turaev Topological Quantum Field Theories \cite{ReTu, Turaevbook} provide  families of finite dimensional projective representations of mapping class groups. 
For each semi-simple Lie algebra, there is an associated theory and an infinite family of such representations.
In this article we are concerned with the 
 $SO(3)$-theory and we will follow the skein-theoretic framework given by Blanchet, Habegger, Masbaum and Vogel \cite{BHMV2}: For each odd integer $r\geqslant 3,$  let $U_r=\lbrace 0,2,4,\ldots, r-3\rbrace$ be the set of even integers smaller than $r-2.$  Given a primitive $2r$-th root of unity $\zeta_{2r},$  a compact oriented surface $\Sigma,$ and a coloring $c$ of the  components of $\partial \Sigma$ by elements of $U_r,$ a finite dimensional $\C$-vector space $RT_r(\Sigma,c)$ is constructed in \cite{BHMV2}, as well as a projective representation:
$$\rho_{r,c} : \mathrm{Mod}(\Sigma) \rightarrow \mathbb{P}\mathrm{Aut}(RT_r(\Sigma,c)).$$
 For
 different choices of root of unity,  the traces of  $\rho_{r,c}$, that are of particular interest to us in this paper,
 are related by actions of Galois groups of cyclotomic fields.
 Unless otherwise indicated, we will always choose $\zeta_{2r}=e^{\frac{i\pi}{r}},$ which is important for us in order to apply
results from \cite{DK, DKY}. 

The representation $\rho_{r,c}$ is called the $SO(3)$-quantum representation of $\mathrm{Mod}(\Sigma)$ at level $r.$ 
Although the representations  are known to be asymptotically faithful \cite{FZW, Andersen}, the question of how well these representations reflect the 
geometry of the mapping class groups remains wide open.

By the  Nielsen-Thurston classification, mapping classes $f\in \mathrm{Mod}(\Sigma)$ are divided into three types: periodic, reducible and
pseudo-Anosov.  Furthermore, the type of $f$ determines the geometric structure, in the sense of Thurston, of the 3-manifold obtained as mapping torus of $f.$
In \cite{AMU} Andersen, Masbaum and Ueno formulated the following conjecture and proved it  when $\Sigma$ is the four-holed sphere.

\begin{conjecture}\label{AMU}{\rm{(AMU conjecture \cite{AMU})}}{ Let $\phi \in  \mathrm{Mod}(\Sigma)$ be a pseudo-Anosov mapping class. Then for any big enough level $r,$ there is a choice of colors $c$ of the components of $\partial \Sigma,$ such that $\rho_{r,c}(\phi)$ has infinite order.}
\end{conjecture}

Note that it is known that the representations $\rho_{r,c}$  send Dehn twists to elements of finite order and criteria for recognizing reducible mapping classes from their images under  $\rho_{r,c}$
are given in
 \cite{Andersen2}.

The results of \cite{AMU}  were extended by   Egsgaard and Jorgensen in \cite{EgsJorgr} and by Santharoubane in \cite{San17}
to prove Conjecture \ref{AMU} for some mapping classes of spheres with $n\geqslant 5$ holes.
In \cite{San12},
Santharoubane  proved the conjecture for 
the one-holed torus.
However, until recently there were no known
cases of the AMU conjecture for mapping classes of surfaces of genus at least $2.$ In \cite{MarSan}, March\'e and Santharoubane used skein theoretic techniques in 
$\Sigma \times S^1$ to obtain  such examples of mapping classes in arbitrary high genus. As explained by Koberda  and Santharoubane \cite{KS}, by means of
Birman exact sequences of mapping class groups, one extracts representations of  $\pi_1(\Sigma)$ from the representations  $\rho_{r,c}.$
Elements in $\pi_1(\Sigma)$ that correspond to pseudo-Anosov mappings classes via Birman exact sequences are characterized by a result of Kra \cite{Kra}.
March\'e and Santharoubane used this approach to obtain their examples of pseudo-Anosov mappings classes satisfying the AMU conjecture by  exhibiting elements in $\pi_1(\Sigma)$ satisfying an additional  technical condition they called Euler incompressibility. However, they informed us that they suspect their construction yields only finitely many mapping classes in any surface of fixed genus, up to mapping class group action.

The purpose of the present paper is to describe an alternative 
method  for approaching the AMU conjecture and use it to construct  mapping classes  
acting on surfaces of any genus, that satisfy the conjecture.  In particular, we produce infinitely many non-conjugate mapping classes acting on surfaces of fixed genus that satisfy the conjecture. 
Our approach  is to relate the conjecture with a question on the growth rate, with respect to $r$, of the $SO(3)$-Turaev-Viro 3-manifold invariants $TV_r$.

 For $M$ a compact orientable $3$-manifold, closed or with boundary, the invariants $TV_r(M)$ are real-valued topological invariants of $M,$ that can be computed from state sums over triangulations of $M$ and are closely related to the $SO(3)$-Witten-Reshetikhin-Turaev TQFTs.
For a compact 3-manifold $M$ (closed or with boundary) we define:
$$lTV(M)=\underset{r \rightarrow \infty, \ r \ \textrm{odd}}{\liminf} \frac{2\pi}{r}\log |TV_r(M,q)|,$$
where $q=\zeta_{2r}=e^{\frac{2i\pi}{r}}$.

Let $f \in \mathrm{Mod}(\Sigma)$ be a mapping class represented by a pseudo-Anosov homeomorphism  of $\Sigma$ and 
let $M_f=F \times [0,1]/_{(x,1)\sim ({ {f}}(x),0)}$ be the mapping torus of $f$.

\begin{theorem}\label{amu-ltv}Let $f \in \mathrm{Mod}(\Sigma)$ be a  pseudo-Anosov  mapping class and 
let $M_{f}$ be the mapping torus of $f$. If $lTV(M_{f})>0,$ then $f$ satisfies the conclusion of the AMU conjecture.
\end{theorem}

The proof of the theorem relies heavily on the properties of  TQFT underlying the  Witten-Reshetikhin-Turaev $SO(3)$-theory as developed in \cite{BHMV2}.

As a consequence of Theorem \ref{amu-ltv} whenever we have a hyperbolic 3-manifold $M$ with $lTV(M)>0$ that fibers over the circle, then the monodromy of the fibration represents a mapping class
that satisfies the AMU conjecture.

By a theorem of Thurston, a mapping class $f\in \mathrm{Mod}(\Sigma)$  is represented by a pseudo-Anosov  homeomorphism of $\Sigma$ 
if and only if the mapping torus $M_{f}$ is hyperbolic. 
In  \cite{Chen-Yang}  Chen and Yang conjectured that for any hyperbolic 3-manifold with finite volume  $M$ we should have $lTV(M)=\vol (M)$. Their conjecture implies, in particular, that
 the aforementioned  technical condition $lTV(M_f)>0$ is  true for all  pseudo-Anosov  mapping classes $f\in  \mathrm{Mod}(\Sigma)$.  Hence, the Chen-Yang conjecture implies the AMU conjecture.

Our method can also be used to produce new families of mapping classes acting on punctured spheres that satisfy the AMU conjecture (see Remark \ref{spheres}). 
In this paper we will be concerned with surfaces with boundary and mapping classes that appear as monodromies of fibered links in $S^3.$
In \cite{BDKY}, with Belletti and Yang, we construct  families of 3-manifolds  in which the monodromies of all hyperbolic fibered links  satisfy the AMU conjecture.
In this paper show the following.
\begin{theorem} \label{hyperbgeneral}  Let  $L\subset S^3$ be a link with $lTV(S^3\setminus L)>0.$ 
Then 
 there are  fibered hyperbolic links  $L',$ with $L\subset L'$ and $lTV(S^3\setminus L')>0,$  and
 such that the complement of $L'$ fibers over $S^1$ with fiber a surface of arbitrarily large genus.
 In particular, the monodromy of such a fibration
  gives a mapping class 
in $\mathrm{Mod}(\Sigma)$ that satisfies the AMU conjecture.
\end{theorem}

In   \cite{DK} the authors gave criteria for constructing  3-manifolds, and in particular link complements,  whose $SO(3)$-Turaev-Viro invariants satisfy the condition $lTV>0$.
Starting from these links, and applying 
Theorem \ref{hyperbgeneral}, we obtain fibered links  whose monodromies 
give examples of mapping classes that satisfy Conjecture \ref{AMU}.
However, the construction yields only finitely many mapping classes in the mapping class groups of fixed surfaces.
This is because the links $L'$ obtained by Theorem \ref{hyperbgeneral}
are represented by  closed homogeneous braids  and it is known that there are only finitely many links of  fixed genus and number of components represented that way.
To obtain infinitely many mapping classes for  surfaces of fixed genus and number of boundary components, we need to refine our construction.
We do this
 by using Stallings twists and appealing  to a result of Long and Morton \cite{LongMorton}
on compositions of pseudo-Anosov maps with powers of a Dehn twist. The general process is given in Theorem \ref{infinitegen}. As an application we have the following.

\begin{theorem} \label{general} Let $\Sigma$ denote an orientable surface of genus $g$ and with $n$-boundary components. Suppose that either $n=2$ and $g\geqslant 3$ or $g\geqslant n \geqslant 3.$
Then  there are are infinitely many non-conjugate pseudo-Anosov mapping classes in
$\mathrm{Mod}(\Sigma)$ that satisfy the AMU conjecture.
\end{theorem}

In the last section of the paper we discuss integrality properties of quantum representations for mapping classes of finite order (i.e. periodic mapping classes) and how they reflect on the Turaev-Viro invariants of the corresponding mapping tori. To state our result, we recall that the traces of the representations $\rho_{r,c}$ are known to be algebraic numbers. For periodic mapping classes we have
 the following.
\begin{theorem}\label{thm:integertrace} Let  $f\in \mathrm{Mod}(\Sigma)$ 
be periodic of order $N.$ For any odd integer $r\geqslant 3$, with $\mathrm{gcd}(r, N)=1$,   we have 
$|\Tr \rho_{r,c}(f)| \in \Z,$
 for any $U_r$-coloring $c$ of $\partial \Sigma,$ and any primitive
 $2r$-root of unity.
\end{theorem}

As a consequence of  Theorem \ref{thm:integertrace}  we have the following corollary that was  conjectured by Chen and Yang \cite[Conjecture 5.1]{Chen-Yang}.
\begin{corollary}\label{cor:toruslinks} For integers $p,q$ let $T_{p,q}$ denote the $(p,q)$-torus link. Then, for any odd $r$ coprime with $p$ and $q$,
we have $TV_r(S^3\setminus T_{p,q})\in \Z$.
\end{corollary}

The paper is organized as follows:  In Section \ref{sec:TV}, we summarize results from the $SO(3)$-Witten-Reshetikhin-Turaev TQFT and their relation to Turaev-Viro invariants that we need in this paper.
In Section  \ref{sec:lTV>0}, we discuss how to construct families of links whose  $SO(3)$-Turaev-Viro invariants have exponential growth (i.e. $lTV>0$) and then we prove Theorem \ref{amu-ltv}
that explains how this exponential growth relates to the AMU Conjecture.
In Section \ref{sec:homogenize}, we describe a method to get hyperbolic fibered links with any given sublink and we prove 
Theorem \ref{hyperbgeneral}.
In Section \ref{sec:examples}, we explain how to  refine the construction of Section \ref{sec:homogenize} to get infinite families of mapping classes on fixed genus surfaces
that satisfy the AMU Conjecture (see Theorem \ref{general}). We also provide an explicit construction that leads to Theorem \ref{general}.

Finally in Section \ref{sec:integertraces}, we discuss periodic mapping classes and we  prove Theorem \ref{thm:integertrace} and Corollary \ref{cor:toruslinks}. We also state a non-integrality conjecture about Turaev-Viro invariants of hyperbolic mapping tori.

\subsection{Acknowledgement} We would like to thank Matthew Hedden, Julien March\'e, Gregor Masbaum, Ramanujan Santharoubane and Tian Yang for helpful discussions. In particular, we thank Tian Yang for bringing to our attention the connection between the volume conjecture of \cite{Chen-Yang} and  the AMU conjecture. We also thank Jorgen Andersen and Alan Reid for their interest in this work. 

\section{TQFT properties and quantum representations} 
\label{sec:TV}

In this section, we summarize some properties of the $SO(3)$-Witten-Reshetikhin-Turaev TQFTs, which we introduce in the skein-theoretic framework of \cite{BHMV2}, and briefly discuss their relation
to the $SO(3)$-Turaev-Viro invariants.

\subsection{Witten-Reshetikhin-Turaev $SO(3)$-TQFTs} Given an
odd  $r\geqslant 3,$  let $U_r$  denote the set of even integers less than $r-2.$ A banded link in a manifold $M$ is an embedding of a disjoint union of annuli $S^1\times[0,1]$ in $M,$ and a $U_r$-colored banded link  $(L,c)$ is a banded link whose components are colored by elements of $U_r.$  
For a closed, oriented 3-manifold $M,$ the Reshetikhin-Turaev
invariants $RT_r(M)$ are complex valued topological invariants.  They also extend to invariants $RT_r(M,(L,c))$ of manifolds containing colored banded links.
These invariants are part of a compatible set of invariants of compact  surfaces and compact 3-manifolds,  which is called a TQFT.
Below we summarize  the main properties of the theory that will be useful to us in this paper, referring the reader to \cite{BHMV, BHMV2} for the  precise definitions and details.

\begin{theorem}\label{thm:TQFTdef}{\rm{ (\cite[Theorem 1.4]{BHMV2})}} For any odd integer $r \geqslant 3$ and any  primitive $2r$-th root of unity $\zeta_{2r},$ there is a TQFT functor $RT_r$ with the following properties:
\begin{itemize}
\item[(1)]  For $\Sigma$ a compact oriented surface, and if  $\partial \Sigma\neq \emptyset$
a coloring $c$ of $\partial \Sigma$ by elements of $U_r$, there is a finite dimensional $\C$-vector space $RT_r(\Sigma,c),$ with a Hermitian form $\langle , \rangle.$
Moreover for disjoint unions, we have $$RT_r(\Sigma \coprod \Sigma')=RT_r(\Sigma)\otimes RT_r(\Sigma').$$
\item[(2)]For $M$ a closed compact oriented $3$-manifold, containing a $U_r$-colored banded link $(L,c),$  the value $RT_r(M,(L,c),\zeta_{2r}) \in \QQ[\zeta_{2r}]\subset \C$ is the $SO(3)$-Reshetikhin-Turaev invariant at level $r.$
\vskip 0.06in

\item[(3)] For $M$ a compact oriented $3$-manifold with $\partial M=\Sigma$, and $(L,c)$ a $U_r$-colored banded link in $M,$ the invariant $RT_r(M,(L,c))$ is a vector in $RT_r(\Sigma).$
 Moreover, for compact oriented 3-manifolds $M_1$, $M_2$ with $\partial M_1=-\partial M_2=\Sigma,$ we have  $$RT_r(M_1\underset{\Sigma}{\cup}M_2)=\langle RT_r(M_1),RT_r(M_2)\rangle.$$
Finally, for disjoint unions $M=M_1\coprod M_2,$ we have
$$RT_r(M)=RT_r(M_1)\otimes RT_r(M_2).$$

\item[(4) ]For a cobordism $M$ with $\partial M= -\Sigma_0 \cup \Sigma_1,$  there is a map
 $$RT_r(M) \in \mathrm{End}(RT_r(\Sigma_0),RT_r(\Sigma_1)).$$
\item[(5)]The composition of cobordisms is sent by $RT_r$ to the composition of linear maps, up to a power of $\zeta_{2r}.$
\end{itemize}
\end{theorem}

In  \cite{BHMV2} the authors construct some explicit orthogonal  basis $E_r$ for  $RT_r(\Sigma,c)$:  Let $\Sigma$ be a compact, oriented surface that is not the 2-torus or the 2-sphere with less than four
holes. Let  $P$ be a collection of simple closed curves  on $\Sigma$ that contains the boundary $\partial \Sigma$
and gives a pants decomposition of $\Sigma.$ The elements of $E_r$ are in one-to-one correspondence with colorings ${\hat{c}}: P \longrightarrow U_r$, such that ${\hat{c}}$ agrees with $c$ on $\partial \Sigma$
and for each pant the colors of the three boundary components satisfy certain admissibility conditions. We will not make use of the general construction. What we need is the following:

\begin{theorem}\label{thm:TQFTbasis} {\rm{(\cite[Theorem 4.11, Corollary 4.10]{BHMV2}) }} 
\\
\item[(1)] For $\Sigma$ a compact, oriented  surface, with genus $g$ and $n$ boundary components, such that $(g,n)\neq (1,0),(0,0),(0,1),(0,2),(0,3),$ we have
$$\mathrm{dim}(RT_r(\Sigma,c))\leqslant r^{3g-3+n}.$$

\item[(2)] If $\Sigma=T$ is the 2-torus we actually have an orthonormal basis for
$RT_r(T).$  It consists of  the elements $e_0,e_2,\ldots,e_{r-3},$ where 
$$e_i=RT_r(D^2\times S^1, ([0,\frac{1}{2}]\times S^1, i))$$
is the Reshetikhin-Turaev vector of the solid torus with the core viewed as banded link and colored by $i.$
\end{theorem}

\subsection{$SO(3)$-quantum representations of the mapping class groups} 
For any odd integer $r \geqslant 3,$ any choice of a primitive $2r$-root of unity $\zeta_{2r}$ and a coloring $c$ of the boundary components of $\Sigma$ by elements of $U_r,$ we have a finite dimensional projective representation,
$$\rho_{r,c} : \mathrm{Mod}(\Sigma) \rightarrow \mathbb{P}\mathrm{Aut}(RT_r(\Sigma,c)).$$
If $\Sigma$ is a closed surface and $f \in \mathrm{Mod}(\Sigma),$ we simply have $\rho_r(f)=RT_r(C_{\phi}),$ where the cobordism $C_{\phi}$ is the mapping cylinder of $f:$
$$C_{f}=\Sigma\times [0,1]\underset{(1,x)\sim f(x)}{\coprod} \Sigma.$$ 
The fact that this gives a projective representation of $\mathrm{Mod}(\Sigma)$ is a consequence of points (4) and (5) of Theorem \ref{thm:TQFTdef}.

For $\Sigma$ with non-empty boundary, giving the precise definition of  the quantum representations would require us to discuss the functor $RT_r$ for cobordisms containing colored tangles (see \cite{BHMV2}). 
 Since in this paper we will only be interested in the traces of the quantum representations, we will not recall the definition of the quantum representations in its full generality. We will use the following theorem:

\begin{theorem}\label{thm:tracequantumrep}
For $r\geqslant 3$ odd, let   $\Sigma$ be a  compact oriented  surface  with $c$  a $U_r$-coloring on the components of $\partial \Sigma.$ Let $\tilde{\Sigma}$ be the surface obtained from $\Sigma$ by capping the components  of $\partial \Sigma$ with disks. For $f\in \mathrm{Mod}(\Sigma),$  let $\tilde{f}\in \mathrm{Mod}(\tilde{\Sigma)},$ denote the mapping class of the extension of $f$ on the capping disks  by the identity.
Let  $M_{\tilde {f}}=F \times [0,1]/_{(x,1)\sim ({\tilde {f}}(x),0)}$
 be the mapping torus of $\tilde{f}$ and let $L \subset M_{\tilde {f}}$ denote the link whose components consist of the cores of the solid tori in $M_{\tilde {f}}$ over the capping disks.
Then, we have
$$\Tr (\rho_{r,c}(f))=RT_r(M_{\tilde {f}} ,(L,c)).$$
\end{theorem}

\subsection{SO(3)-Turaev-Viro invariants}
In \cite{TuraevViro}, Turaev and Viro introduced invariants of compact oriented $3$-manifolds as state sums on triangulations of 3-manifolds.
The triangulations are colored by representations of a semi-simple quantum Lie algebra. In this paper, we are only concerned with the $SO(3)$-theory: Given a compact 3-manifold $M$, an odd integer $r\geqslant 3,$
and a primitive $2r$-root of unity, 
there is an $\RR$-valued invariant $TV_r.$ We refer to \cite{DKY} for the precise flavor of Turaev-Viro invariants we are using here, and to \cite{TuraevViro} and for the original definitions and proofs of invariance.
We will make use the following theorem, which relates the Turaev-Viro invariants $TV_r(M)$ of a $3$-manifold $M$ with the Witten-Reshetikhin-Turaev TQFT $RT_r.$  For closed 3-manifolds it was proved by
Roberts \cite{Roberts}, and was extended to manifolds with boundary by Benedetti and Petronio \cite{BePe}. In fact, as Benedetti and Petronio formulated their theorem in the case of $\mathrm{SU}_2$-TQFT, the adaptation of the proof in  the setting of $SO(3)$-TQFT  we use here can be found in \cite{DKY}.
\begin{theorem}{\rm{ (\cite[Theorem 3.2]{BePe})}}\label{thm:BePe}
For $M$ an oriented compact $3$-manifold with empty or toroidal boundary and $r\geqslant 3$ an odd integer, we have:
$$TV_r(M,q=e^{\frac{2i\pi}{r}})=||RT_r(M,\zeta=e^{\frac{i\pi}{r}})||^2.$$
\end{theorem}


\section{Growth of Turaev-Viro invariants and the AMU conjecture}

In this section,  first we explain how the growth of the  $SO(3)$-Turaev-Viro invariants is related to the AMU conjecture.
Then we give examples of link complements $M$  for which the $SO(3)$-Turaev-Viro invariants have exponential growth with respect to $r$; that is, we have $lTV(M)>0.$

\subsection{Exponential growth implies the AMU conjecture}
Let $\Sigma$ denote a compact  orientable surface with or without boundary and, as before, let $\mathrm{Mod}(\Sigma)$  denote the mapping class group of $\Sigma$ fixing the boundary.

\begin{named}{Theorem  \ref{amu-ltv}}   Let $f \in \mathrm{Mod}(\Sigma)$ be a  pseudo-Anosov  mapping class and 
let $M_{f}$ be the mapping torus of $f$. If $lTV(M_{f})>0,$ then $f$ satisfies the conclusion of the AMU conjecture.
\end{named}
\label{sec:amu-ltv}
The proof of Theorem \ref{amu-ltv} relies on the following elementary lemma:
\begin{lemma}\label{lem:order} If $A \in \mathrm{GL}_n (\mathbb{C})$ is such that $|\mathrm{Tr}(A)|>n,$ then $A$ has infinite order.
\end{lemma}
\begin{proof}
Up to conjugation we can assume that $A$
is upper triangular. If the sum of the $n$ diagonal entries has modulus bigger than $n,$ one of these entries must have modulus bigger that $1$. This implies that
$A$ has infinite order.
\end{proof}

\vskip 0.06in
\begin{proof}[Proof of Theorem \ref{amu-ltv}]  Suppose that for the mapping torus  $M_{f}$ of some $f \in \mathrm{Mod}(\Sigma)$, we  have $lTV(M_{f})>0.$ We will prove Theorem \ref{amu-ltv} by relating $TV_r(M_{f})$ to traces of the quantum representations of $\mathrm{Mod}(\Sigma)$.
By Theorem \ref{thm:BePe}, we have $$TV_r(M_{f})=||RT_r(M_{f})||^2= \langle RT_r(M_{f}), \  RT_r(M_{f})\rangle,$$ 
where, with the notation of Theorem  \ref{thm:TQFTdef},  $ \langle , \rangle$ is the   Hermitian form on  $RT_r(\Sigma,c).$ 

Suppose that $\Sigma$ has genus $g$ and $n$ boundary components.
Now $\partial M_f$ is a disjoint union of $n$  tori. Note that by Theorem \ref{thm:TQFTbasis}-(2) and Theorem \ref{thm:TQFTdef}-(1), $RT_r(\partial M_f)$ admits an orthonormal basis given by vectors
$${\bf e}_{c}=e_{c_1}\otimes e_{c_2} \otimes \ldots e_{c_n},$$
where $c=(c_1,c_2,\ldots c_n)$ runs over all $n$-tuples of colors in $U_r,$ one for each boundary component. By Theorem \ref{thm:TQFTdef}-(3) and Theorem \ref{thm:TQFTbasis}-(2), this vector is also the $RT_r$-vector of the cobordism consisting of $n$ solid tori, with the $i$-th solid torus containing the core colored by $c_i.$

We can write  $ RT_r(M_{f})=\underset{c}{\sum} \lambda_c {\bf e}_{c}$ where $\lambda_c=\langle RT_r(M_f), {\bf e}_{c}\rangle$. Thus we have

$$TV_r(M_{f})=\underset{c}{\sum}|\lambda_c|^2=\underset{c}{\sum} |\langle RT_r(M_f), {\bf e}_{c}\rangle|^2,$$
where ${\bf e}_c$ is the above orthonormal basis of $RT_r(\partial M_{f})$ and the sum runs over $n$-tuples of colors in $U_r.$  
 By Theorem \ref{thm:TQFTdef}-(3), the pairing $\langle RT_r(M_{f}),{\bf  e}_{c}\rangle$ is obtained by filling the boundary components of $M_{f}$ by solid tori and adding a link $L$ which is the union of the cores and the core of the $i$-th component is colored by $c_i.$
Thus by Theorem \ref{thm:tracequantumrep}, we have
$$\langle RT_r(M_{f}), {\bf e}_{c} \rangle= RT_r({M_{\tilde{f}}},(L,c))=\Tr (\rho_{r,c}(f)),$$
and thus
$$TV_r(M_{f})=\underset{c}{\sum}|\mathrm{Tr}\rho_{r,c}(f)|^2,$$
where the sum ranges over all colorings of the boundary components of $M_{f}$ by elements of $U_r.$
Now, on the one hand, since $lTV(M_{f})>0,$ the sequence $\{TV_r(M_{f})\}_r$ is bounded below by a sequence that is exponentially growing in $r$ as $r \rightarrow \infty.$ 
On the other hand, by Theorem \ref{thm:TQFTbasis}-(1), the sequence  $\underset{c}{\sum}\mathrm{dim}(RT_r(\Sigma,c))$ only grows polynomially in $r.$
 For big enough $r,$ there will be at least one $c$ such that $|\mathrm{Tr}\rho_{r,c}(f)|>\mathrm{dim}(RT_r(\Sigma,c))$. Thus by Lemma \ref{lem:order}, $\rho_{r,c}(\phi)$ will have infinite order.
\end{proof}

By a theorem of Thurston \cite{thurston:mappingtori}, a mapping class $f \in \mathrm{Mod}(\Sigma)$  is represented by a pseudo-Anosov  homeomorphism of $\Sigma$ 
if and only if the mapping torus $M_f$ is hyperbolic. 

As a consequence of Theorem \ref{amu-ltv}, whenever a hyperbolic 3-manifold $M$ that fibers over the circle has $lTV(M)>0,$ the monodromy of the fibration represents a mapping class
that satisfies the AMU Conjecture.

In the remaining of this paper we will be concerned with surfaces with boundary and mapping classes that appear as monodromies of fibered links in $S^3.$

\subsection{Link complements with $lTV>0$}
\label{sec:lTV>0}
Links with exponentially growing Turaev-Viro invariants will be the fundamental building block of our construction of examples of pseudo-Anosov mapping classes satisfying the AMU conjecture.
We will need the following result proved by the authors in
\cite{DK}.

\begin{theorem}\label{thm:ltvdehnfilling}{\rm {(\cite[Corollary 5.3]{DK})}} Assume that $M$ and $M'$ are oriented compact 3-manifolds with empty or toroidal  boundaries and such that
 $M$ is obtained by a Dehn-filling of $M'.$ Then we have:
 $$lTV(M)\leqslant lTV(M').$$
 
\end{theorem}

Note that for a link  $L\subset S^3$, and a sublink $K\subset L$, the complement of $K$ is obtained from that of $L$ by Dehn-filling.
Thus Theorem \ref{thm:ltvdehnfilling} implies that if $K$ is a sublink of a link $L\subset S^3$ and  $lTV(S^3 \setminus K)>0,$ then we have $lTV(S^3 \setminus L) )>0.$

\begin{corollary} \label{positive} Let $K\subset S^3$ be the knot $4_1$ or a link with complement homeomorphic to that of the Boromean links or the Whitehead link. 
If $L$ is any link containing $K$ as a sublink then  $lTV(S^3 \setminus L) )>0.$
\end{corollary}
\begin{proof}Denote by $B$ the Borromean rings. By \cite{DKY}, $lTV(S^3\setminus 4_1)=2v_3\simeq 2.02988$ and $lTV(S^3\setminus B)=2v_8\simeq 7.32772;$
and hence the conclusion holds for $B$ and $4_1.$
 The complement of $K=4_1$ is obtained by Dehn filing along one of  the components of the Whitehead link $W.$
Thus, by Theorem \ref{thm:ltvdehnfilling},  $lTV(S^3\setminus W)\geqslant 2 v_3>0.$ 
 For links with homeomorphic complements the conclusion follows since the Turaev-Viro invariants
are homeomorphism invariants of the link complement; that is, they will not distinguish different links with homeomorphic complements.
\end{proof}

\begin{remark}Additional classes of links with $lTV>0$ are given by the authors in \cite{DKY} and \cite{DK}. Some of these examples are non-hyperbolic. 
However it is known that any link is a sublink of a hyperbolic link \cite{Baker}. Thus one can start with any link $K$ with $lTV(S^3\setminus K)>0$ and construct hyperbolic links  $L$ containing $K$ as sublink; by Theorem \ref{thm:ltvdehnfilling} these will still have $lTV(S^3 \setminus L)> 0.$
\end{remark}


\section{A hyperbolic version of Stallings's homogenization}
\label{sec:homogenize}
A classical result of Stallings \cite{Stallings} states that every link $L$ is a sublink of fibered links  with fibers of arbitrarily large genera. Our purpose in this section is to prove the following hyperbolic version of this result.

\begin{theorem} \label{mainofsection} Given a link $L\subset S^3,$ 
 there are  hyperbolic links  $L',$ with $L\subset L'$ and
 such that the complement of $L'$ fibers over $S^1$ with fiber a surface of arbitrarily large genus.
\end{theorem}
\subsection{Homogeneous braids}
Let $\sigma_1,\ldots, \sigma_{n-1}$ denote the standard braid generators of the $n$-strings braid group $B_n$.
We recall that a braid $\sigma \in B_n$ is said to be \emph{homogeneous} if each standard generator $\sigma_i$ appearing in $\sigma$ always appears  with exponents of the same sign. 
In \cite{Stallings}, Stallings studied relations between closed homogeneous braids and fibered links. We summarize his results as follows:

\begin{enumerate}

\item The closure of any homogeneous braid $\sigma \in B_n$ is a fibered link: The complement fibers over $S^1$ with fiber the surface $F$ obtained by Seifert's algorithm from the homogeneous closed braid
diagram. The Euler characteristic of $F$ is $\chi(F)=n-c(\sigma)$ where $c(\sigma)$ is the number of  crossings of $\sigma.$ 

\item Given a link $L=\hat{\sigma}$ represented as the closure of a braid $\sigma \in B_n,$  one can add additional strands to obtain a homogeneous braid $\sigma' \in B_{n+k}$
so that the closure of $\sigma'$ is a link $L\cup K,$ where $K,$ the closure of the additional $k$-strands, represents the unknot. Furthermore, we can arrange $\sigma'$ so that the linking numbers of $K$ with the components of $\hat{\sigma}$ are any arbitrary numbers. The link $L\cup K,$ as a closed homogeneous braid, is fibered.
\\

Throughout the paper we will refer to the component $K$ of  $L\cup K,$ as the Stallings component.
\end{enumerate}

In order to prove Theorem \ref{mainofsection},
given a hyperbolic link $L$, we want to apply Stallings' homogenizing method in a way such that the resulting link is still hyperbolic.

Let $L$ be a hyperbolic link with $n$ components  $L_1,\ldots, L_n.$  The complement $M_L:=S^3\setminus L$ is a hyperbolic 3-manifold with $n$
cusps; one for each component.  For each cusp, corresponding to some component  $L_i$,  there is a conjugacy class of a rank two abelian subgroup of $\pi_1(M_L).$ We will refer to this as  the {\emph{peripheral group}} of $L_i.$

\begin{definition} \label{defcondition}Let $L$ be a hyperbolic link with $n$ components  $L_1,\ldots, L_n.$ We say that an unknotted circle $K$ embedded in $S^3 \setminus L$ satisfies condition $(\clubsuit )$ if
(i)  the free homotopy class $[K]$ does not lie in a peripheral group of any component of $L$; and (ii)  we have 
$$\mathrm{gcd}\left(lk(K,L_1),lk(K,L_2),\ldots , lk(K,L_n)\right)=1.$$ 
\end{definition}

The rest of this subsection is devoted to the proof of the following proposition that is needed for the proof of Theorem \ref{mainofsection}.

\begin{proposition}\label{prop:condition} Given a hyperbolic link  $L,$ one can choose the Stallings component $K$ so that  (i) $K$ satisfies condition $(\clubsuit )$;  and (ii)
the fiber of the complement of $L\cup K$ has arbitrarily high genus.
\end{proposition}

Since  $L$ is hyperbolic, we have the 
 discrete faithful representation 
 
 $$\rho:  \pi_1(S^3\setminus L)\longrightarrow \mathrm{PSL}_2(\C).$$
 We recall that an element of $A\in  \mathrm{PSL}_2(\C)$ is called {\emph{parabolic}} if  $\Tr (A)=\pm 2,$
 and that $\rho$ takes  elements  in the peripheral subgroups of $ \pi_1(S^3 \setminus L)$ to parabolic elements  in $ \mathrm{PSL}_2(\C).$
 Since matrix trace is invariant under conjugation, in the discussion below we will not make distinction between elements in $\pi_1(M_L)$ and their conjugacy classes.
 With this understanding we recall that
  if an element $\gamma \in \pi_1(S^3 \setminus L)$ satisfies $\Tr (\rho (\gamma))\neq \pm 2,$ then it does not lie in any peripheral subgroup \cite[Chapter 5]{thurston:notes}.

\begin{lemma}\label{lem:traces} Let $A$ and $B$ be elements in $\mathrm{PSL}_2(\C).$
\begin{itemize}
\item[(1)]If $A$ and $B$ are non commuting parabolic elements  then $|\Tr( A^l B^{-l})| >2$ for some $l.$
\item[(2)]If $|\Tr(A)|>2,$ then $|\Tr (A^k B)|\neq 2$ for all $k$ big enough.
\end{itemize}
\begin{proof}
For (1), note that after conjugation we can take $A=\begin{pmatrix}
1 & 0 \\ 1 & 1
\end{pmatrix}$ and $B=\begin{pmatrix}
1 & x \\ 0 & 1
\end{pmatrix}.$ Then $$|\Tr(A^l B^{-l})|=|1-l^2 x| \underset{l\rightarrow\infty}{\rightarrow} \infty.$$
For (2),  after conjugation take $A=\begin{pmatrix}
\lambda & 0 \\ 0 & \lambda^{-1}
\end{pmatrix}$ where $|\lambda|>1$ and write $B=\begin{pmatrix}
 u & v \\ w & x
\end{pmatrix}.$ Then 
 $$\Tr (A^k B)=\lambda^k u+\lambda^{-k} x,$$
 which as $k \rightarrow +\infty$ tends either to infinity if $u\neq 0$ or to $0$ else. In the first case we will have
 $|\Tr (A^k B)|> 2$  for $k$ big enough; in the second case we will have 
$|\Tr (A^k B)|< 2$. In both cases we have  $|\Tr (A^k B)|\neq 2$ as desired.
\end{proof}
\end{lemma}

Next we will consider the Wirtinger presentation of $\pi_1(S^3 \setminus L)$ corresponding to a link diagram 
representing a
hyperbolic link $L$  as a closed braid 
$\hat{\sigma}.$  The Wirtinger generators are conjugates of meridians of the components of $L$ and are mapped to parabolic elements of $\mathrm{PSL}_2(\C)$ by $\rho$.
A key point in the proof of Proposition \ref{prop:condition} is to choose  the Stallings component $K$  so that the word it represents in
$\pi_1(S^3 \setminus L)$ is conjugate to one that  begins with a sub-word
 $(a^l b^{-l})^k,$ where $a$ and $b$ are Wirtinger generators mapped to non-commuting  elements under $\rho$. Then we will use
 Lemma \ref{lem:traces} to prove that the free homotopy class $[K]$ is not in a peripheral subgroup of any component of $L$.
 We first need the following lemma:
 
\begin{lemma}\label{lem:meridians} Let $L$ be a hyperbolic link in $S^3,$ with a link diagram of a closed braid 
$\hat{\sigma}.$ We can find two strands of $\sigma$ meeting at a crossing so that if $a$ and $b$ are the Wirtinger generators corresponding to an under-strand and the  over-strand of  the crossing respectively, then
$\rho(a)$ and $\rho(b)$ don't commute.

\end{lemma}
\begin{proof}
Suppose that for any pair of  Wirtinger  generators $a, b$ corresponding to a crossing as above, $\rho(a)$ and $\rho(b)$ commute.
Since  $\rho(a)$ and $\rho(b)$  are commuting parabolic elements of infinite order in $\mathrm{PSL}_2(\C),$ elementary linear algebra shows that they share 
 their unique eigenline. Then step by step, we get that the images under $\rho$ of all Wirtinger generators share an eigenline. But this would imply that $\rho\left(\pi_1(S^3\setminus L)\right)$ is abelian
  which is a contradiction.
\end{proof}
We can now turn to the proof of Proposition \ref{prop:condition}, which we will prove by tweaking Stallings homogenization procedure. 

\begin{proof}[Proof of Proposition \ref{prop:condition}]  Let $L$ be a hyperbolic link, with components, 
 $L_1,\ldots,L_n$, represented as a braid closure ${\hat{ \sigma}}$. Let $a,b$ be Wirtinger generators of $\pi_1(S^3\setminus L)$ chosen as in 
 Lemma \ref{lem:meridians}.
 
 Starting with the projection of ${\hat{ \sigma}}$, we proceed 
in the following way:

We arrange the crossings of ${\hat{ \sigma}}$ to occur at different verticals  on the projection plane.
\begin{enumerate}

\item Begin drawing the Stallings component so that near the strands where above chosen Wirtinger generators $a,b$ occur,
we 
create the pattern shown the left of Figure \ref{fig:pattern}. 

\item We deform the strands of $\sigma$ to create
`` zigzags"  as shown in the second drawing of Figure \ref{fig:homogenization}. 

\item We fill the empty spaces in verticals with new braid strands 
and choose the new crossings so that the resulting braid is homogeneous and so that the new strands  meet the strands  of $\sigma$ both in positive and negative crossings. 
Adding enough `` zigzags" at the previous step will ensure that there is enough freedom in choosing the crossings to make this second condition possible.

\item  At this stage, we have turned the braid $\sigma$ into a homogeneous braid, say $\sigma_h$. The closure ${\hat{ \sigma_h}}$  contains $L$ as a sublink and some number $s\geqslant 1$ of unknotted components.
To reduce the number of components added, we connect the new components with a single crossing between each pair of neighboring new components. Doing so we may have to create new crossings with the components of $L,$ but we can always choose them to preserve homogeneousness. Thus we homogenized $L$ by adding a single unknotted component $K$ to it. 
\end{enumerate}
The four step process described above  is illustrated in Figure \ref{fig:homogenization}.
\begin{figure}
  \centering
    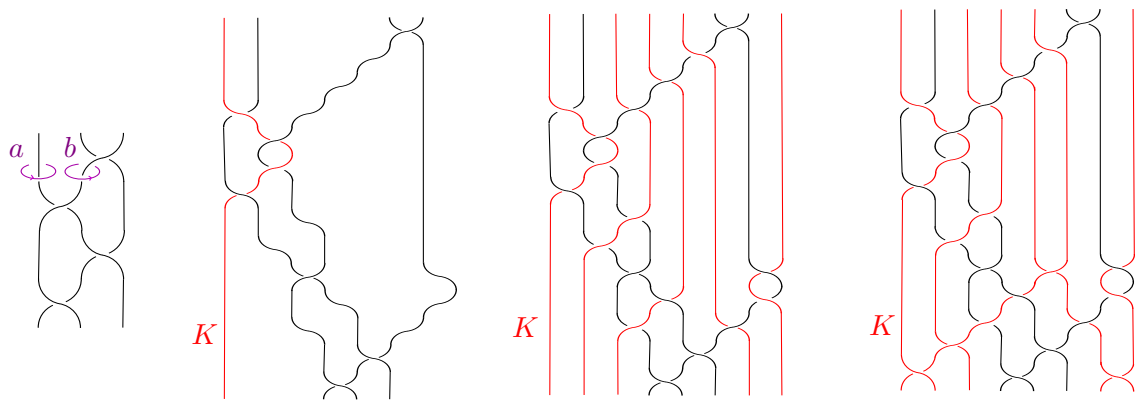
  \caption{The four step homogenization process.}
  \label{fig:homogenization}
\end{figure}

Now, because we have positive and negative crossings of $K$ with each component of $L,$ we can set the linking numbers as we want just by adding an even number of positive or negative crossings between $K$ and a component of $L$ locally.
If the strands $a$ and $b$ correspond to the same component $L_1,$ we simply ask that $lk(K,L_1)=1.$
If they correspond to two distinct components $L_1$ and $L_2,$ we choose $(lk(K,L_1),lk(K,L_2))=(1,0).$

Recall that we have chosen $a,b$ to be Wirtinger generators of $\pi_1(S^3\setminus L)$, as  in 
 Lemma \ref{lem:meridians}, and so
that  $K$ is added to $L$
so that  the pattern shown on the left hand side of Figure \ref{fig:pattern} occurs near the corresponding crossing.
Assume that $[K]$ is conjugate to a word $w \in \pi_1(S^3\setminus L).$  Now one may modify the diagram of $L\cup K$ locally, as shown in the right hand side of Figure \ref{fig:pattern},
to make $[K]$ conjugate to $(a^{-l} b^l)^k w$ for any non-negative $k$ and $l.$  Notice also this move leaves $K$ unknotted and that $L \cup K$  is still a closed homogeneous braid.
 Also notice  that doing so, we left $lk(K,L_1)$ unchanged if $a$ and $b$ were part of the same component $L_1,$ and we turned $(lk(K,L_1),lk(K,L_2))$ into $(1-kl,kl)$ if they correspond to different components $L_1$ and $L_2.$ In both cases, we preserved the fact that 
  $$gcd(lk(K,L_1),lk(K,L_2),\ldots ,lk(K,L_n))=1$$
 and $K$ satisfies part (ii) of Condition $(\clubsuit).$

\begin{figure}
   \centering  
  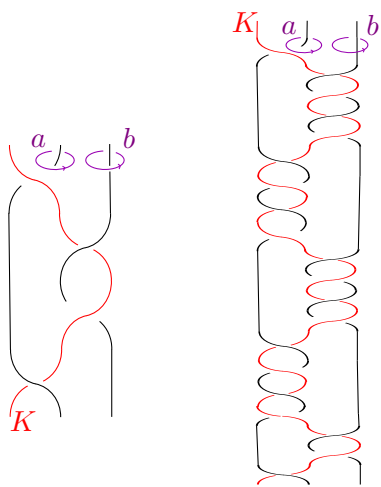
  \caption{Changing  $[K]$ from a conjugate of  $w \in \pi_1(S^3\setminus L)$ (left), to a conjugate of $w(a^{-l}b^l)^k$ (right), for any non-negative $l$ and $k.$ Here $l=k=2.$}
  \label{fig:pattern}
\end{figure}

To ensure that part (i) of the condition is satisfied, note that since $\rho(a)$ and $\rho(b)$ are non-commuting, Lemma \ref{lem:traces} (1)  implies $|\Tr (\rho(a)^{-l}\rho(b)^l)|> 2$ for $l>>0.$ Thus by choosing $k>>0,$ and using   Lemma \ref{lem:traces} (2), we may assume that $|\Tr(A^k B)|\neq 2,$ where $A=\rho(a)^{-l} \rho(b)^{l}$ and $B=\rho(w).$ Then
$[K]=A^k B$ is not in a peripheral subgroup of $\pi_1(S^3\setminus L).$

Now notice that as above mentioned positive integers $k,l$ become arbitrarily large, the crossing number of the resulting homogeneous braid projections becomes arbitrarily large while the braid index remains unchanged. Since the fiber of the fibration of a closed homogeneous braid is the Seifert surface of the closed braid projection it follows that as  $k,l\to \infty$, the genus of the fiber becomes arbitrarily large.
 \end{proof}
 
\subsection{Ensuring hyperbolicity}

In this subsection we will finish the proof of Theorem \ref{mainofsection}. For this we need the following:

\begin{proposition}\label{prop:hyperbolic}  Suppose that $L$ is a hyperbolic link and let $L\cup K$ be a  homogeneous closed  braid obtained from $L$ by adding a Stallings 
component  $K$  that satisfies condition $(\clubsuit )$.  Then $L\cup K$ is a hyperbolic link.
\end{proposition}

Before we can proceed with the proof of Proposition \ref{prop:hyperbolic} we need some preparation:
We recall that when an oriented link $L$ is embedded in a solid torus, the total winding number of $L$ is the non-negative integer $n$ such that $L$ represents $n$ times a generator of $H_1(V,\Z).$
When convenient we will  consider  $M_{L\cup K}$ to be the compact 3-manifold obtained by removing the interiors of neighborhoods of the components of $L\cup K$; the interior of $M_{L\cup K}$
is homeomorphic to $S^3 \setminus (L\cup K)$.
In the course of the proof of the proposition we will
see that condition $(\clubsuit)$ ensures that the complement of $L\cup K$ cannot contain embedded tori that are not boundary parallel or compressible (i.e.  $M_{L\cup K}$  is {\emph{atoroidal}}).
We need the following lemma that provides restrictions on winding numbers of satellite fibered links.

\begin{lemma}\label{lem:windingnumber}We have the following:

\begin{itemize}
\item[(1)] Suppose that $L$ is a oriented fibered link in $S^3$ that is embedded in a solid torus $V$ with boundary $T$ incompressible in $S^3\setminus L.$ Then, some component of $L$ must have non-zero
 winding number.
\item[(2)] Suppose that $L$ is an oriented fibered link in $S^3$ such that only one component $K$  is embedded inside a solid torus $V.$ If $K$ has   winding number $1,$ then $K$ is isotopic to the core of $V.$
\end{itemize}
\end{lemma}
Though this statement is fairly classical in the context of fibered knots \cite{HiraMuraSilver}, we include a proof as we are working with fibered links.
\begin{proof} The complement $M_L=S^3\setminus N(L)$ fibers over $S^1$ with fiber a surface $(F, \partial F)\subset (M_L, \partial M_L)$. Then $S^3\setminus L$ cut along $F=F\times \{0\}=F\times \{1\}$ is homeomorphic to $F \times [0,1].$ 
It is known that $F$ maximizes the Euler characteristic in its homology class in $H_2(M_L, \partial M_L)$ and thus $F$ is incompressible and $\partial$-incompressible.
\smallskip

(1) Assume that the winding number of every component of $L$  is zero, and consider the intersection of $F$ with $T,$ the boundary of the solid torus containing $L.$  Since $F \times [0,1]$ is irreducible, and $F$ is incompressible in the complement of $L$, up to isotopy, one can assume 
that the intersection $T\cap F$ consists of a collection of parallel curves in $T,$  each of which is homotopically essential in $T.$
The hypothesis on the winding number implies that
the intersection $F\cap T$ is null-homologous in $T,$ where each component of  $F\cap T$  is given the orientation inherited by the surface $V\cap F.$
Thus the curves in $F\cap T$ can be partitioned in pairs 
of  parallel curves with opposite orientations in $T \cap F$. Each such pair bounds an annulus in $T$ and 
in $F \times (0,1)$ each of these annuli has both ends on $F \times \lbrace 0 \rbrace$ or on $F \times \lbrace 1 \rbrace.$  This implies that we can find $0<t<1$ such that
$F_t=F\times \{t\}$ misses the torus $T$. This in turn implies that  $T$ must be an essential torus in the manifold obtained by cutting $S^3\setminus L$ along the fiber $F_t$.
But this is impossible since the later manifold is $F_t\times I$ which is a handlebody and cannot contain essential tori; contradiction.

\smallskip

(2) By an argument similar to that used in case (1) above, we can simplify the intersection of the fiber surface $F$ with $T$ until it consists of one curve only. This curve, say  $\gamma$, cuts $T$ into an essential annulus embedded in $F \times (0,1)$ with one boundary component on $F \times \lbrace 0 \rbrace$ and the other on $F \times \lbrace 1 \rbrace.$ As the annulus closes up, the curve $\gamma$ must be fixed by the monodromy of the fibration and one can isotope $T$ to make it compatible with the fibration. Then one has that $K$ is fibered in $V,$ and as the winding number of $K$ is $1,$ by Corollary 1 in \cite{HiraMuraSilver}, $K$ must be isotopic to the core of $V.$ 
\end{proof}

We are now ready to  give the proof of Proposition \ref{prop:hyperbolic}.

\begin{proof}[Proof of Proposition \ref{prop:hyperbolic}] First we remark that $S^3 \setminus (L \cup K)$ is non-split as $S^3 \setminus L$ is and $K$ represents a non-trivial element in $\pi_1(S^3 \setminus L).$

Next we argue that $S^3 \setminus (L \cup K)$ is atoroidal:  Assume that we have an essential torus, say  $T,$ in $M_{L\cup K}=S^3 \setminus (L \cup K).$ Since $L$ is hyperbolic, in $M_L=S^3 \setminus L$
the torus $T$  becomes either boundary parallel or compressible. Moreover, the torus $T$ bounds a solid torus $V$ in $S^3.$

Suppose that $T$ becomes boundary parallel in the complement of $L$. Then, we may assume that $V$ is a tubular neighborhood of a component $L_i$ of $V.$
Then $K$ must lie inside $V$; for otherwise $T$ would still be boundary parallel in $M_{L\cup K}$.
Then the free homotopy class $[K]$ would represent a conjugacy class in the peripheral subgroup of $\pi_1(M_L)$ corresponding to $L_i.$ However this contradicts condition $(\clubsuit);$ thus this case cannot happen.

Suppose now that we know that $T$ becomes compressible in $M_L$.
In $S^3,$  the torus $T$ bounds a solid torus $V$ that contains a compressing disk of $T$ in $M_L.$ If $V$ contains no component of $L \cup K,$ the torus $T$ is still compressible in $M_{L\cup K}$. Otherwise, there are again two cases: 
 \smallskip
 
 {\it  Case 1:}  The solid torus $V$ contains some components of $L.$  We claim that $V$ actually contains all the components of $L.$ Otherwise, after compressing $T$ in $M_L,$ one would get a sphere that separates the components of $L,$ which can not happen as $L$ is non-split. Moreover, as the compressing disk is inside $V,$ all components of $L$ have winding number zero in $V.$

Since $T$ is incompressible in the complement of $M_{L\cup K},$ the component $K$ must also lie inside $V.$
Note that $V$ has to be knotted since otherwise $T$ would compress outside $V$ and thus in $M_{L\cup K}$.  But then since $K$ is unknotted, it must have winding number zero in $V.$
Thus we have the fibered link $L\cup K$ lying inside $V$ so that each component has winding number zero. But then $T$ can not be incompressible in $M_{L\cup K}$ by Lemma \ref{lem:windingnumber}-(1); contradiction. Thus this case will not happen.

\smallskip

 {\it  Case 2:}   The solid torus $V$ contains only $K.$  Since $T$ is incompressible in $M_{L\cup K},$  $K$ must be geometrically essential in $V;$ that is it doesn't lie
 in a 3-ball inside $V.$ Since $K$ is unknotted, it follows that $V$ is unknotted.
 For each component $L_i$ of $L,$ we have
 $$lk(K,L_i)=w\cdot  lk(c,L_i),$$
where $c$ is the core of $V,$ 
 and $w$ denote the winding number of $K$ in $V.$  Since $K$ satisfies condition $(\clubsuit),$ we know that 
  $$gcd\left(lk(K,L_1),lk(K,L_2),\ldots ,lk(K,L_n)\right)=1,$$
which implies that we must  have  $w=1.$
  Thus by Lemma \ref{lem:windingnumber}-(2), $K$ is isotopic to the core of $V$ and $T$ is boundary parallel, contradicting the assumption that $T$ is essential in $M_{L\cup K}.$
This finishes the proof that  $M_{L\cup K}$ is atoroidal.

 Since $M_{L\cup K}$ contains no essential spheres or tori, and has toroidal boundary, it is either a Seifert  fibered space or a hyperbolic manifold. But $M_L$ is a Dehn-filling of $M_{L\cup K}$ which is hyperbolic. Since the Gromov norm $||\cdot ||$  does not increase under Dehn filling \cite{thurston:notes} we get  $||M_{L\cup K}||\geq ||M_L||>0.$ The Gromov norm of Seifert 3-manifolds is zero, thus $L \cup K$ must be hyperbolic.
\end{proof}

\smallskip

We can now finish the proof of Theorem \ref{mainofsection} and the proof of Theorem \ref{hyperbgeneral} stated in the Introduction.

\begin{proof}[Proof of Theorem  \ref{mainofsection}]  Let $L$ be any link.
If $L$ is not hyperbolic, then we can find a hyperbolic link $L',$ that contains $L$ as a sub-link. See for example \cite{Baker}. If $L$ is hyperbolic then set $L=L'.$ 
Then apply 
Proposition \ref{prop:condition} to $L'$ to get  links $L'\cup K$ that are closed homogeneous braids with arbitrarily high crossing numbers and fixed braid index.
By \cite{Stallings}, the links $L'\cup K$  are fibered and  the fibers have arbitrarily large genus and by Proposition  \ref{prop:hyperbolic}
they are hyperbolic.
\end{proof}

\begin{proof}[Proof of Theorem    \ref{hyperbgeneral}] Suppose that $L$ is a link with $lTV(S^3\setminus L)>0$. By Theorem  \ref{mainofsection} we have fibered hyperbolic links $L'$ that contain $L$ as sublink and whose fibers have
arbitrarily large genus. By Theorem \ref{thm:ltvdehnfilling} we have $lTV(S^3\setminus L)>0$. \end{proof}

\section{Stallings twists and the AMU conjecture}
\label{sec:examples}
By our results in the previous sections, starting from a hyperbolic link $L\subset S^3$ with $lTV(S^3\setminus L) >0,$ one can add  an unknotted component $K$ to obtain a hyperbolic fibered link  $L\cup K,$
with $lTV(S^3\setminus (L\cup K)) >0.$
 The monodromy of a fibration of $L\cup K$ provides a  pseudo-Anosov mapping class
on the surface $\Sigma=\Sigma_{g,n},$  where $g$ is the genus of the fiber and $n$ is the number of components of $L\cup K.$

One can always increase the number of boundary components $n$ by adding more components to $L$ and appealing to  Theorem \ref{thm:ltvdehnfilling}. However since $L\cup K$ is a closed homogeneous braid this construction alone will not provide infinite families of examples for fixed genus and number of boundary components.

In this section we show how to address this problem and prove Theorem \ref{general} stated in the introduction and which, for the convenience of the reader we restate here.

\begin{named}{Theorem \ref{general}} Let $\Sigma$ denote an orientable surface of genus $g$ and with $n$-boundary components. Suppose that either $n=2$ and $g\geqslant 3$ or $g\geqslant n \geqslant 3.$
Then  there are are infinitely many non-conjugate pseudo-Anosov mapping classes in 
$\mathrm{Mod}(\Sigma)$ that satisfy the AMU conjecture.
\end{named}

\subsection{Stallings twists and pseudo-Anosov mappings}

Stallings \cite{Stallings}  introduced an operation that transforms a fibered link into a fibered link with a fiber of the same genus:
 Let $L$ be a fibered link with 
fiber $F$ and let $c$ be a simple closed curve on $F$ that is unknotted in $S^3$ and such that $lk(c,c^+)=0,$ where $c^+$ is the curve $c$ pushed along the normal of $F$ in the positive direction.
The curve $c$ bounds a disk  $D\subset S^3$ that is transverse to $F$. Let $L_m$ denote the link obtained from $L$ by a full twist of order $m$ along $D.$ This operation is known as Stallings twist of order $m.$
 Alternatively, one can think the Stallings twist operation as performing $1/m$ surgery on $c,$ where the framing of $c$ is induced by the normal vector on $F.$

\begin{theorem} {\rm { (\cite[Theorem 4]{Stallings})}}\label{twist}  Let $L$ be a link whose complement fibers over $S^1$ with fiber $F$ and monodromy $f$.
Let $L_m$ denote a link obtained by a Stallings twist of order $m$ along a curve $c$ on $F$. Then, the complement of $L_m$ fibers over $S^1$ with fiber $F$
and the monodromy is $f \circ \tau_c^m,$ where $\tau_c$ is the Dehn-twist on $F$ along $c.$ 
\end{theorem}

Note that when $c$ is parallel to a component of $L,$ then such an operation does not change the homeomorphism class of the link complement; we call these Stallings twists trivial.

To facilitate the identification of non-trivial Stallings twists on link fibers, we recall the notion of {\emph {state graphs}}:

  Recall that the fiber for the complement of a homogeneous closed braid ${\hat{\sigma}}$ is obtained as follows: Resolve all the crossings in the projection of ${\hat{\sigma}}$ in a way consistent with the braid orientation.
The result is a collection of nested embedded circles (Seifert circles) each bounding a disk on the projection plane; the disks can be made disjoint by pushing them slightly above the projection plane. Then we construct the fiber $F$ by attaching a half twisted band for each crossing.
The state graph consists of the collection of the Seifert circles together with an edge for each crossing of ${\hat{\sigma}}.$  We will label each edge by $A$ or $B$ according to whether 
the resolution of the corresponding crossing during the construction of $F$ is of type $A$ or $B$ shown in  Figure \ref{fig:resolutions}, if viewed as unoriented resolution.

\begin{figure}
   \centering  
  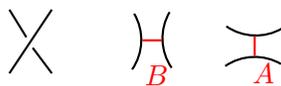
  \caption{A crossing and its $A$ and $B$ resolutions.}
  \label{fig:resolutions}
\end{figure}
\begin{figure}
   \centering  
  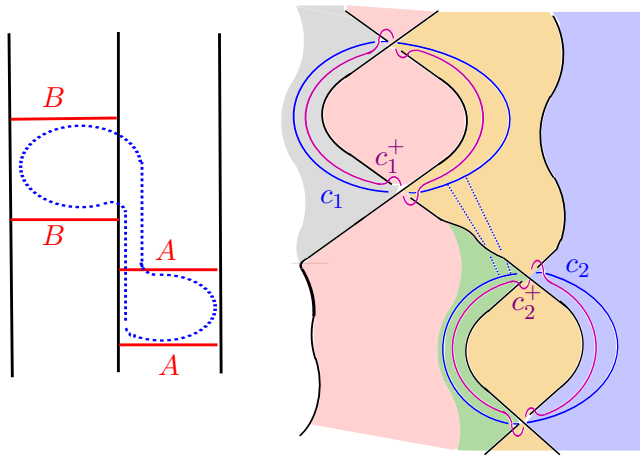
  \caption{Left: Pattern in the state graph exhibiting  a non-trivial  Stallings twist. Right: The curve $c$ obtained as a connected sum the curves $c_1, c_2$ with self linking $+2$ and $-2$. }
  \label{fig:stallingstwist}
\end{figure}

\begin{remark}\label{locate} {\rm As the homogeneous braids get more complicated the fiber is  more likely to admit a non-trivial Stallings twist. Indeed, if the state graph of $L={\hat{\sigma}}$ exhibits the  local pattern  
shown in the left hand side of 
 Figure \ref{fig:stallingstwist}, we can perform a non-trivial Stallings twist along the curve $c$ which corresponds to the connected sum of the two curves $c_1$ and $c_2$ shown in the Figure. We can see that $lk(c_1,c_1^+)=+2$ and $lk(c_2,c_2^+)=-2,$ and the mixed linkings are zero. In the end, $lk(c,c+)=2-2=0.$}
 \end{remark}

We will need the following theorem, stated and proved  by Long and Morton \cite{LongMorton} for closed surfaces. Here we state the bounded version and for completeness we sketch the slight adaptation of their argument in this setting.
 
 \begin{theorem} {\rm {(\cite[Theorem A]{LongMorton})}}\ \label{thm:LongMorton}Let $F$ be a compact oriented surface with $\partial F\neq 0.$  Let  $f$ be 
 a pseudo-Anosov homeomorphism on $F$ and let $c$ be a non-trivial, non-boundary parallel simple closed curve on $F.$ 
 Let $\tau_{c}$ denote the Dehn-twist along $c.$ Then, the family $\{f \circ \tau_{c}^m\}_m$  contains infinitely  many non-conjugate pseudo-Anosov homeomorphisms.
 \end{theorem}
\begin{proof} 
 The proof rests on the fact that the mapping torus of $f_m = f \circ \tau_{c}^m$ is obtained from $M_f$ by performing 
$1/m$-surgery on the curve $c$ with framing induced by a normal vector in $F.$ Once we prove that $M_f \setminus c$ is hyperbolic,
Thurston's hyperbolic Dehn surgery theorem implies, for $m$  big enough, that
the mapping tori $M_{f_{m}}$ are hyperbolic and all pairwise non-homeomorphic (as their hyperbolic volumes differ). Since conjugate maps have homeomorphic mapping tori the non-finiteness statement follows.

We will consider the curve $c$ as embedded on the fiber $F \times \lbrace 1/2 \rbrace \subset M_f.$ 
Notice that $M_f \setminus c$ is irreducible, as $c$ is non-trivial in $\pi_1(F)$ and thus in $\pi_1(M_f).$

 We need to show that $M_f \setminus c$ contains no essential embedded tori: Let $T$ be a torus embedded in $M_f \setminus c.$
 If $T$ is boundary parallel in $M_f,$ it will also be in $M_f \setminus c,$ otherwise one would be able to isotope $c$ onto the boundary of $M_f,$ and as $c$ is actually a curve on $F \times \lbrace 1/2 \rbrace,$ $c$ would be conjugate in 
 $\pi_1(F)$ to a boundary component. As we chose $c$ non-boundary parallel in $F,$ this does not happen.

Now, assume that $T$ is non-boundary parallel in $M_f.$ 
Then we can put $T$ in general position and consider of $T\cap F\times \lbrace 1/2 \rbrace.$ If this intersection is empty then $T$ is compressible as $F \times [0, 1]$ does not contain any essential tori. After  isotopy we can assume  that $T \cap F \times [0, 1]$ is a collection of properly embedded annuli in $ F \times [0, 1],$
each of which either misses a fiber $F$ or is vertical with respect to the $I$-bundle structure.
Now note that if one of these annuli misses a fiber then we can remove it by isotopy in $M_f \setminus c,$  unless if it connects to curves parallel to $c$ on opposite sides of $c$ on 
$ F\times \lbrace 1/2 \rbrace.$  Also observe that we cannot have annuli that connect a non-boundary parallel curve $c \subset F\times \lbrace 1/2 \rbrace$
to $f(c):$  For, since $f$ is pseudo-Anosov,  the curves 
$f^k (c)$ and $f^l (c)$ are freely homotopic on the fiber  if and only if $k=l;$ and thus the annuli would never close up to give $T.$
In the end, and since $M_f$ is hyperbolic, 
 we are left with two annuli connecting both sides of $c$ and $T$ is boundary parallel in $M_f \setminus c.$
\\ Finally, $M_f\setminus c$ is irreducible and atoroidal and since its Gromov norm  satisfies $||M_f \setminus c||>||M_f||>0,$  it is hyperbolic.\end{proof}

\subsection{Infinite families of mapping classes} We are now ready to present our examples of infinite families of non-conjugate pseudo-Anosov mapping classes of fixed surfaces that satisfy the AMU conjecture.
The following theorem gives the general process of the construction.
\begin{theorem} \label{infinitegen} Let $L$ be a hyperbolic fibered link with fiber $\Sigma$  and monodromy $f.$ Suppose that
$L$ contains a sublink $K$ with $lTV(S^3\setminus K)>0.$ Suppose, moreover, that the fiber $\Sigma$ admits a non-trivial Stallings twist along a curve $c\subset \Sigma$ such that the interior of the twisting disc $D$ intersects $K$ at most once geometrically. Let $\tau_c$ denote the Dehn twist of $\Sigma$ along $c$.
Then the family  $\{f \circ \tau_{c}^m\}_m$  of homeomorphisms gives   infinitely many non-conjugate pseudo-Anosov  mappings classes in $\mathrm{Mod}(\Sigma)$ that satisfy the AMU conjecture.
\end{theorem}
\begin{proof} Since $L$  contains $K$ as sublink we have $lTV(S^3\setminus L)\geqslant lTV(S^3\setminus K)>0.$ 
Since $D$ intersects $K$ at most once, each of the links $L'$ obtained by Stalling twists along $c$,  will also contain a sublink isotopic to $K$ and hence  $lTV(S^3\setminus L')\geqslant lTV(S^3\setminus K)>0.$ 
The conclusion follows by Theorems \ref{twist}, \ref{thm:LongMorton} and \ref{amu-ltv}.
\end{proof}

We finish the section with concrete constructions of infinite families obtained by applying Theorem \ref{infinitegen}. Start with $K_1=4_1$ represented as the closure of the homogeneous braid
$\sigma_2^{-1} \sigma_1 \sigma_2^{-1} \sigma_1.$
We construction a 2-parameter family of links 
$L_{n,m}$ where $n\geqslant 2$, $m\geqslant 1,$ defined as follows:

The link $L_{4,m}$ 
 is shown in the left panel of Figure \ref{fig:example2},  where the box shown contains $2m$ crossings.  It is obtained from $K_1$ by adding three unknotted components.

 The link  $L_{3,m}$  is obtained from $L_{4,m}$  by removing the unknotted component corresponding to the outermost string of the braid. 
 
 The link $L_{n+1,m}$ for $n\geqslant 4$ is obtained from $L_{n,m}$ adding one strand in the following way: denote by $K_1,\ldots ,K_n$ the components of $L_{n,m}$ from innermost to outermost, $K_1$ being the $4_1$ component.
 To get $L_{n+1,m},$ we add one strand $K_{n+1}$ to $L_{n,m}$, so that traveling along $K_n$ one finds $2$ crossings with $K_{n-1},$ then $2$ crossings with $K_{n+1},$ then $2$ crossings with $K_{n-1},$ then $2$ crossings with $K_{n+1},$ and, moreover, the crossings with $K_{n-1}$ and $K_{n+1}$ have opposite signs. There is only one way to chose this new strand, and doing so we added one unknotted component to $L_{n,m},$ thus $L_{n+1,m}$ has $n+1$ components and $4$ more crossings than $L_{n,m}.$  
 \\ In the special case $n=2,$ the link $L_{2,m}$ is obtained from the link $L_{3,m}$ by replacing the box with $2m$ crossings with a box with $2m-1$ crossings. The links $L_{2,m}$ are then $2$-components links. We note that all the links $L_{n,m}$ contain the component $K_1$ we started with.

\begin{figure}
\centering
  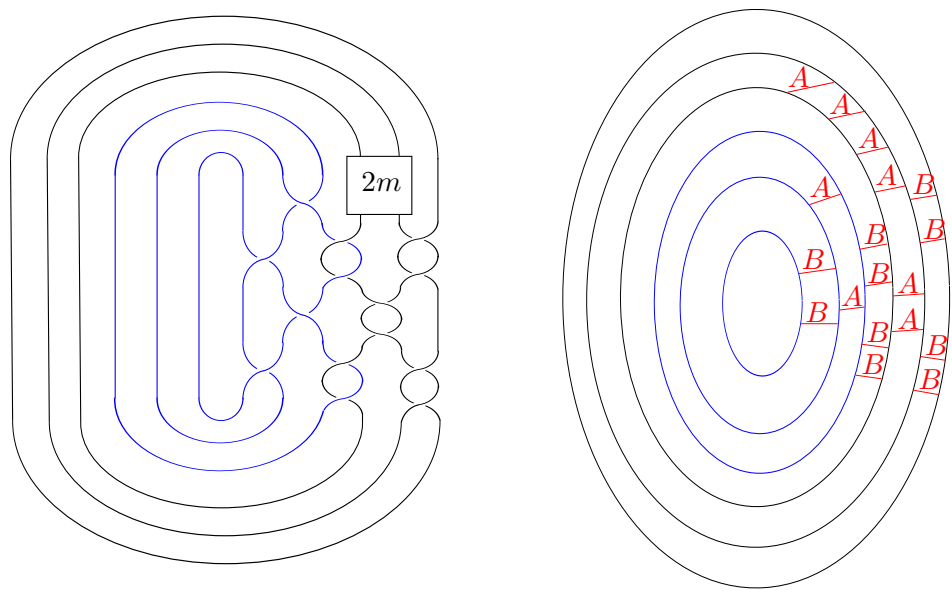
  \caption{The link $L_{4,m}$ and the  state graph for $L_{4,2}.$}
  \label{fig:example2}
\end{figure} 
\begin{proposition}\label{prop:example} The link
$L_{n,m}$ is hyperbolic, fibered and satisfies the hypotheses of Theorem \ref{infinitegen}.
The fiber has genus $g=m+2$ if $n=2,$
 $g=n+m-1$ otherwise.
\end{proposition}
\begin{proof} For every $n\geqslant 2$ and any $m\geqslant 1,$ the link $L_{n,m}$ contains  the knot $K_1=4_1$ as sublink  and as said earlier we have $lTV(S^3\setminus K_1)>0.$
Since $L_{n,m}$ is alternating, hyperbolicity follows from Menasco's criterion \cite{menasco:primediagram}: any prime non-split alternating diagram of a link that is not the standard diagram of the $T(2,q)$ torus link,  represents a hyperbolic link.
Since $L_{n,m}$ is represented by an alternating (and thus homogeneous) closed braid, fiberedness follows from Stalling's criterion. 
For $n\geqslant 3,$ the resulting closed braid  diagram 
has  braid index $2+n$ and $2+4n+2m$ crossings. Hence the Euler characteristic is $-3n-2m$ and the genus is $m+n-1,$ as the fiber has $n$ boundary components. In the case $n=2,$ the braid index is $5,$ number of crossings $9+2m,$ thus the Euler characteristic of the fiber is $-4-2m$ and the genus is $m+2.$
 Using Remark \ref{locate} and the state graph given in  Figure \ref{fig:example2} we can easily locate a simple closed curve  $c$ on the fiber with the properties
in the statement of Theorem \ref{infinitegen}.
\end{proof}

Now Proposition  \ref{prop:example} and Theorem \ref{infinitegen} immediately give Theorem \ref{general} stated in the beginning of the section.

\begin{remark} Note that if we restrict ourselves to closed homogeneous braids to get explicit examples of links satisfying the hypotheses of Theorem \ref{infinitegen}, it seems necessary that the genus will grow with the number of components. However, one could consider other methods, that can increase the number of components, while keeping the fiber genus low,  to produce links satisfying the hypotheses of Theorem \ref{infinitegen}. One way would be to take a Murasugi sum of the link $L_{2,m}$ with links of arbitrarily large number of components and whose complement fibers over $S^1$ with fiber of genus $0$.
This should be done so that the Murasugi sum operation leaves the component $4_1\subset L_{2,m}$ unaffected and it produces hyperbolic links. It seems plausible  that combining homogeneous braids and Murasugi sums should give explicit examples of infinite families of mapping classes that satisfy the AMU Conjecture, for all  surfaces $\Sigma$ with genus $g\geqslant 2$ and $n\geqslant 2$ boundary components. 
\end{remark}

\begin{remark}\label {spheres} Our methods also apply to surfaces of genus zero to produce examples of mapping class that satisfy the conclusion of the AMU Conjecture.
Given that such examples were previously known, we just outline an explicit construction without pursuing the details of determining the Nielsen-Thurston types of the resulting mapping classes or discussing how the construction relates to that of Santharoubane \cite{San17}:
Let $\Sigma_{0, n}$ denote the sphere with $n$ holes. A mapping class in $Mod(\Sigma_{0, n+1})$ can be thought as an element in the pure braid group on $n$-strings, say $P_n$.
It is known that for $n>3$,  $P_n$ is a semi-direct product  of a subgroup $W_n$ that is itself a semi direct product of free subgroups of $P_n$ and of $P_3$. See, for example,
\cite[Theorem 1.8]{birman:book}. As result any braid $b\in P_n$ can be uniquely  written as a product $\beta\cdot w$ where 
$\beta\in P_3$ and $w\in W_n$. Now take  any $b=\beta \cdot w$ for which the closure of $\beta\in P_3$ represents the Borromean rings $B$.
The braid $b$ represents an element in $Mod(\Sigma_{0, n+1})$ whose mapping torus is the complement of the link $L_b$ that is the closure of $b$ together with the braid axis.
The link $L_b$ contains $B$ as a sublink and hence $lTV(S^3\setminus L_b)>lTV(S^3\setminus B)>0.$ Thus by Theorem \ref{amu-ltv}, for $r$ big enough,  
there is a choice of colors $c$ of the components of $\Sigma_{0, n+1}$ such that $\rho_{r,c}(\phi)$ has infinite order. 
\end{remark}
\begin{remark} \label{knots}Theorem \ref{infinitegen} leads to constructions of mapping classes on surfaces with at least two boundary components that satisfy the AMU. Furthermore, all these mapping classes are obtained 
as monodromies of fibered links in $S^3$. In \cite{BDKY} we prove the Turaev-Viro  invariants volume conjecture for an infinite family of cusped hyperbolic 3-manifolds. Considering the doubles of these 3-manifolds we obtain an infinite family of closed 3-manifolds $M$ with  $lTV(M)>0$. It is known that every closed 3-manifold contains hyperbolic fibered knots \cite{Soma}.
By Theorem \ref{amu-ltv}, monodromies of such knots provide examples of pseudo-Anosov mappings on surfaces with a single boundary component that satisfy  the AMU conjecture.
\end{remark}


\section{Integrality properties of periodic mapping classes}
\label{sec:integertraces}
In this section we give the proofs of Theorem \ref{thm:integertrace} and Corollary \ref{cor:toruslinks} stated in the Introduction. We also state a conjecture about traces of quantum representations of pseudo-Anosov 
mapping classes and we give some supporting evidence.

\begin{named} {Theorem \ref{thm:integertrace}} Let  $f\in \mathrm{Mod}(\Sigma)$ 
be periodic of order $N.$ For any odd integer $r\geqslant 3$, with $\mathrm{gcd}(r, N)=1$,   we have 
$|\Tr \rho_{r,c}(f)| \in \Z,$
 for any $U_r$-coloring $c$ of $\partial \Sigma,$ and any primitive
 $2r$-root of unity.
\end{named}

\begin{proof} For any choice of a primitive  $2r$-root of unity $\zeta_{2r}$, the traces $\Tr \rho_{r,c}(f)$ lie in the field $ \QQ[\zeta_{2r}]$. Since  the $\Z$ is invariant under the action of the Galois group of the field,
the property $|\Tr \rho_{r,c}(f)| \in \Z,$ 
does not depend on the choice of root of unity to define the TQFT.

In the rest of the proof, for any positive integer $n,$ we write $\zeta_n=e^{\frac{2i\pi}{n}}.$ 

By choosing a lift we can  consider $\rho_{r,c}(f)$ as an element of $\mathrm{Aut}(RT_r(\Sigma,c))$ instead of $\mathbb{P}\mathrm{Aut}(RT_r(\Sigma,c)).$ 
Since $f^N=id$ and $\rho_{r,c}$ is a projective representation with projective ambiguity a $2r$-root of unity, we have $\rho_r(f)^N=\zeta_{2r}^k\  \mathrm{Id}_{RT_r(\Sigma,c)}.$  Since $N$ and $r$ are coprime, by changing the lift $\rho_{r,c}(f)$ by a power of $\zeta_{2r}$ we can assume actually that $\rho_{r,c}(f)^N=\pm \mathrm{Id}_{RT_r(\Sigma,c)}.$ Then $\rho_{r,c}(f)$ is diagonalizable, with eigenvalues that are $2N$-th roots of unity. This implies that $|\Tr \rho_{r,c} (f)| \in \Z[\zeta_{2N}].$

On the other hand we know that the traces of quantum representations $\rho_{r, c}$ take values in $\QQ [\zeta_{2r}]$, and the same is true  for $RT_r$ invariants of any closed $3$-manifold with a colored link (see Theorem \ref{thm:TQFTdef}-(2)). Thus we have
$$|\Tr \rho_{r,c}(f)| \in \Z[\zeta_{2N}]\cap \QQ[\zeta_{2r}].$$
By elementary number theory, it is known that
$$\QQ[\zeta_m]\cap \QQ[\zeta_n]= \QQ[\zeta_{d}],$$
where $d=\mathrm{gcd}(m, n)$. See, for example,  \cite[Theorem 3.4]{Konrad} for a proof of this fact. 
 Hence, since $\QQ[\zeta_2]=\QQ[-1]=\QQ$ and the algebraic integers in $\QQ$ are the integers,
 we have $\Z[\zeta_{2N}]\cap \QQ[\zeta_{2r}]=\Z.$ 
Thus we obtain $|\Tr \rho_{r,c}(f)| \in \Z.$ 
\end{proof}

It is known that the mapping torus of a class $f\in \mathrm{Mod}(\Sigma)$ 
is a Seifert fibered manifold if and only if $f$ is periodic. In particular, the complement $S^3\setminus T_{p,q}$ of a $(p, q)$ torus link is fibered with periodic monodromy of order 
$pq$ \cite{orlik:seifert-manifolds}. As a corollary of Theorem \ref{thm:integertrace} we have the following result which in particular implies Corollary \ref{cor:toruslinks}  that settles a question of \cite{Chen-Yang}.

\begin{corollary}\label{integerTV} Let  $M_f$ be the mapping torus of a periodic mapping class $f\in \mathrm{Mod}(\Sigma)$ of order $N$. Then, for  any odd integer $r\geqslant 3$, with $\mathrm{gcd}(r, N)=1$,   we have
$TV_r(M_f)\in \Z,$
 for any choice of root of unity.
\end{corollary}

\begin{proof}

As in the proof of Theorem \ref{amu-ltv}, we write
$$TV_r(M_f)=\underset{\mathbf{c}}{\sum} |\Tr \rho_{r,\mathbf{c}}(f)|^2$$
where $f$ is the monodromy and the sum is over $U_r$-colorings of the components of $\partial \Sigma$.
 But if $r$ is coprime with $N$ this sum is a sum of integers by Theorem \ref{thm:integertrace}.
\end{proof}

Corollary \ref{integerTV} implies that for mapping tori of periodic classes the Turaev-Viro invariants take integer values at infinitely many levels and this property is independent of the choice of
the root of unity. In contrast with this we have the following, were $lTV$ is defined in the Introduction.

\begin{proposition}\label{prop:integertrace} Let $f \in \mathrm{Mod}(\Sigma)$ such that $lTV(M_f)>0.$ 
Then,  there can be at most finitely many odd integers $r$ such that $TV_r(M_f)\in \Z$.
\end{proposition}
\begin{proof} As in the proof of Corollary \ref{integerTV} and Theorem \ref{thm:integertrace} for any odd $r\geq 3$ and any choice of a  primitive $2r$-root of unity $\zeta_{2r}$ 
the invariant  $TV_r(M_f, e^{\frac{2i\pi}{r}})$  lies in $\mathbb F=\QQ [e^{\frac{i\pi}{r}}].$ 

Suppose that there are arbitrarily large odd levels  $r$ such that $TV_r(M_f, e^{\frac{2i\pi}{r}})\in \Z$.
Then since $\Z$ is left fixed under the action of the Galois group of $\mathbb F$, we would have 
$TV_r(M_f, e^{\frac{i\pi}{r}})=TV_r(M_f, e^{\frac{2i\pi}{r}})$, for all $r$ as above.

But this is contradiction: Indeed, on the one hand, the assumption $lTV(M_f)>0,$
implies that the invariants $TV_r(M_f, e^{\frac{2i\pi}{r}})$ grow exponentially in $r$; that is $TV_r(M_f, e^{\frac{2i\pi}{r}}) > \exp{Br}$, for some constant $B>0$.
On the other hand, by combining results of \cite{Garoufalidis} and  \cite{BePe}, the invariants $TV_r(M_f, e^{\frac{i\pi}{r}})$ grow at most polynomially in $r$; that is $TV_r(M_f, e^{\frac{i\pi}{r}})\leqslant D r^N$, for some constants
$D>0$ and $N$. 
\end{proof}

As discussed earlier the Turaev-Viro invariants volume conjecture of \cite{Chen-Yang} implies that  for all pseudo-Anosov  mapping classes we have $lTV(M_f)>0,$ and the later hypothesis implies the AMU Conjecture. These implications and Proposition \ref{prop:integertrace} prompt the following conjecture suggesting that the Turaev-Viro invariants of mapping tori distinguish
pseudo-Anosov mapping classes from periodic ones.

\begin{conjecture}\label{conj:integertrace} Suppose that $f\in \mathrm{Mod}(\Sigma)$  is
pseudo-Anosov. Then, there can be at most finitely many odd integers $r$ such that $TV_r(M_f)\in \Z.$
\end{conjecture} 

\bibliographystyle{hamsplain}
\bibliography{biblio}
\end{document}

%% file: homogenization.pdf_tex
\begingroup%
  \makeatletter%
  \providecommand\color[2][]{%
    \errmessage{(Inkscape) Color is used for the text in Inkscape, but the package 'color.sty' is not loaded}%
    \renewcommand\color[2][]{}%
  }%
  \providecommand\transparent[1]{%
    \errmessage{(Inkscape) Transparency is used (non-zero) for the text in Inkscape, but the package 'transparent.sty' is not loaded}%
    \renewcommand\transparent[1]{}%
  }%
  \providecommand\rotatebox[2]{#2}%
  \ifx\svgwidth\undefined%
    \setlength{\unitlength}{424.21069336bp}%
    \ifx\svgscale\undefined%
      \relax%
    \else%
      \setlength{\unitlength}{\unitlength * \real{\svgscale}}%
    \fi%
  \else%
    \setlength{\unitlength}{\svgwidth}%
  \fi%
  \global\let\svgwidth\undefined%
  \global\let\svgscale\undefined%
  \makeatother%
  \begin{picture}(1,0.34700227)%
    \put(0,0){\includegraphics[width=\unitlength]{homogenization.pdf}}%
    \put(-0.00033727,0.21520935){\color[rgb]{0.50196078,0,0.50196078}\makebox(0,0)[lb]{\smash{$a$}}}%
    \put(0.04866209,0.21570019){\color[rgb]{0.50196078,0,0.50196078}\makebox(0,0)[lb]{\smash{$b$}}}%
    \put(0.16192514,0.04953003){\color[rgb]{1,0,0}\makebox(0,0)[lb]{\smash{$K$}}}%
    \put(0.44671288,0.05521519){\color[rgb]{1,0,0}\makebox(0,0)[lb]{\smash{$K$}}}%
    \put(0.76389564,0.05759646){\color[rgb]{1,0,0}\makebox(0,0)[lb]{\smash{$K$}}}%
  \end{picture}%
\endgroup%

%% file: pattern.pdf_tex
\begingroup%
  \makeatletter%
  \providecommand\color[2][]{%
    \errmessage{(Inkscape) Color is used for the text in Inkscape, but the package 'color.sty' is not loaded}%
    \renewcommand\color[2][]{}%
  }%
  \providecommand\transparent[1]{%
    \errmessage{(Inkscape) Transparency is used (non-zero) for the text in Inkscape, but the package 'transparent.sty' is not loaded}%
    \renewcommand\transparent[1]{}%
  }%
  \providecommand\rotatebox[2]{#2}%
  \ifx\svgwidth\undefined%
    \setlength{\unitlength}{142.22248535bp}%
    \ifx\svgscale\undefined%
      \relax%
    \else%
      \setlength{\unitlength}{\unitlength * \real{\svgscale}}%
    \fi%
  \else%
    \setlength{\unitlength}{\svgwidth}%
  \fi%
  \global\let\svgwidth\undefined%
  \global\let\svgscale\undefined%
  \makeatother%
  \begin{picture}(1,1.2327633)%
    \put(0,0){\includegraphics[width=\unitlength]{pattern.pdf}}%
    \put(0.30311671,0.89437582){\color[rgb]{0.50196078,0,0.50196078}\makebox(0,0)[lb]{\smash{$b$}}}%
    \put(0.05985822,0.89650107){\color[rgb]{0.50196078,0,0.50196078}\makebox(0,0)[lb]{\smash{$a$}}}%
    \put(0.94924253,1.19384129){\color[rgb]{0.50196078,0,0.50196078}\makebox(0,0)[lb]{\smash{$b$}}}%
    \put(0.72376954,1.19558255){\color[rgb]{0.50196078,0,0.50196078}\makebox(0,0)[lb]{\smash{$a$}}}%
    \put(0.00401785,0.14464255){\color[rgb]{1,0,0}\makebox(0,0)[lb]{\smash{$K$}}}%
    \put(0.5926327,1.20133705){\color[rgb]{1,0,0}\makebox(0,0)[lb]{\smash{$K$}}}%
  \end{picture}%
\endgroup%

%% file: resolution.pdf_tex
\begingroup%
  \makeatletter%
  \providecommand\color[2][]{%
    \errmessage{(Inkscape) Color is used for the text in Inkscape, but the package 'color.sty' is not loaded}%
    \renewcommand\color[2][]{}%
  }%
  \providecommand\transparent[1]{%
    \errmessage{(Inkscape) Transparency is used (non-zero) for the text in Inkscape, but the package 'transparent.sty' is not loaded}%
    \renewcommand\transparent[1]{}%
  }%
  \providecommand\rotatebox[2]{#2}%
  \ifx\svgwidth\undefined%
    \setlength{\unitlength}{110.45932617bp}%
    \ifx\svgscale\undefined%
      \relax%
    \else%
      \setlength{\unitlength}{\unitlength * \real{\svgscale}}%
    \fi%
  \else%
    \setlength{\unitlength}{\svgwidth}%
  \fi%
  \global\let\svgwidth\undefined%
  \global\let\svgscale\undefined%
  \makeatother%
  \begin{picture}(1,0.29072467)%
    \put(0,0){\includegraphics[width=\unitlength]{resolution.pdf}}%
    \put(0.46521737,0.03090212){\color[rgb]{1,0,0}\makebox(0,0)[lb]{\smash{$B$}}}%
    \put(0.83288576,0.03621242){\color[rgb]{1,0,0}\makebox(0,0)[lb]{\smash{$A$}}}%
  \end{picture}%
\endgroup%

%% file: stallingstwist.pdf_tex
\begingroup%
  \makeatletter%
  \providecommand\color[2][]{%
    \errmessage{(Inkscape) Color is used for the text in Inkscape, but the package 'color.sty' is not loaded}%
    \renewcommand\color[2][]{}%
  }%
  \providecommand\transparent[1]{%
    \errmessage{(Inkscape) Transparency is used (non-zero) for the text in Inkscape, but the package 'transparent.sty' is not loaded}%
    \renewcommand\transparent[1]{}%
  }%
  \providecommand\rotatebox[2]{#2}%
  \ifx\svgwidth\undefined%
    \setlength{\unitlength}{244.3855957bp}%
    \ifx\svgscale\undefined%
      \relax%
    \else%
      \setlength{\unitlength}{\unitlength * \real{\svgscale}}%
    \fi%
  \else%
    \setlength{\unitlength}{\svgwidth}%
  \fi%
  \global\let\svgwidth\undefined%
  \global\let\svgscale\undefined%
  \makeatother%
  \begin{picture}(1,0.69467826)%
    \put(0,0){\includegraphics[width=\unitlength]{stallingstwist.pdf}}%
    \put(0.06024438,0.53779509){\color[rgb]{1,0,0}\makebox(0,0)[lb]{\smash{$B$}}}%
    \put(0.06024438,0.32529193){\color[rgb]{1,0,0}\makebox(0,0)[lb]{\smash{$B$}}}%
    \put(0.23457378,0.301297){\color[rgb]{1,0,0}\makebox(0,0)[lb]{\smash{$A$}}}%
    \put(0.23967221,0.12447995){\color[rgb]{1,0,0}\makebox(0,0)[lb]{\smash{$A$}}}%
    \put(0.48939902,0.38936982){\color[rgb]{0,0,1}\makebox(0,0)[lb]{\smash{$c_1$}}}%
    \put(0.57372107,0.44889129){\color[rgb]{0.50196078,0,0.50196078}\makebox(0,0)[lb]{\smash{$c_1^+$}}}%
    \put(0.86802186,0.2827271){\color[rgb]{0,0,1}\makebox(0,0)[lb]{\smash{$c_2$}}}%
    \put(0.78452644,0.22403229){\color[rgb]{0.50196078,0,0.50196078}\makebox(0,0)[lb]{\smash{$c_2^+$}}}%
  \end{picture}%
\endgroup%

%% file: Example2.pdf_tex
\begingroup%
  \makeatletter%
  \providecommand\color[2][]{%
    \errmessage{(Inkscape) Color is used for the text in Inkscape, but the package 'color.sty' is not loaded}%
    \renewcommand\color[2][]{}%
  }%
  \providecommand\transparent[1]{%
    \errmessage{(Inkscape) Transparency is used (non-zero) for the text in Inkscape, but the package 'transparent.sty' is not loaded}%
    \renewcommand\transparent[1]{}%
  }%
  \providecommand\rotatebox[2]{#2}%
  \ifx\svgwidth\undefined%
    \setlength{\unitlength}{354.23830566bp}%
    \ifx\svgscale\undefined%
      \relax%
    \else%
      \setlength{\unitlength}{\unitlength * \real{\svgscale}}%
    \fi%
  \else%
    \setlength{\unitlength}{\svgwidth}%
  \fi%
  \global\let\svgwidth\undefined%
  \global\let\svgscale\undefined%
  \makeatother%
  \begin{picture}(1,0.61663756)%
    \put(0,0){\includegraphics[width=\unitlength]{Example2.pdf}}%
    \put(0.37400635,0.42296948){\color[rgb]{0,0,0}\makebox(0,0)[lb]{\smash{$2m$}}}%
    \put(0.85049809,0.41657275){\color[rgb]{1,0,0}\makebox(0,0)[lb]{\smash{$A$}}}%
    \put(0.84611543,0.28392048){\color[rgb]{1,0,0}\makebox(0,0)[lb]{\smash{$B$}}}%
    \put(0.88429254,0.30001616){\color[rgb]{1,0,0}\makebox(0,0)[lb]{\smash{$A$}}}%
    \put(0.84251103,0.33941609){\color[rgb]{1,0,0}\makebox(0,0)[lb]{\smash{$B$}}}%
    \put(0.9120415,0.32480298){\color[rgb]{1,0,0}\makebox(0,0)[lb]{\smash{$B$}}}%
    \put(0.90780985,0.36607721){\color[rgb]{1,0,0}\makebox(0,0)[lb]{\smash{$B$}}}%
    \put(0.95987867,0.41938398){\color[rgb]{1,0,0}\makebox(0,0)[lb]{\smash{$B$}}}%
    \put(0.97120244,0.37194165){\color[rgb]{1,0,0}\makebox(0,0)[lb]{\smash{$B$}}}%
    \put(0.97361203,0.24828835){\color[rgb]{1,0,0}\makebox(0,0)[lb]{\smash{$B$}}}%
    \put(0.96652382,0.21228353){\color[rgb]{1,0,0}\makebox(0,0)[lb]{\smash{$B$}}}%
    \put(0.92089235,0.42914772){\color[rgb]{1,0,0}\makebox(0,0)[lb]{\smash{$A$}}}%
    \put(0.94373342,0.31408446){\color[rgb]{1,0,0}\makebox(0,0)[lb]{\smash{$A$}}}%
    \put(0.94430374,0.27639961){\color[rgb]{1,0,0}\makebox(0,0)[lb]{\smash{$A$}}}%
    \put(0.90191496,0.46832039){\color[rgb]{1,0,0}\makebox(0,0)[lb]{\smash{$A$}}}%
    \put(0.91100645,0.26179148){\color[rgb]{1,0,0}\makebox(0,0)[lb]{\smash{$B$}}}%
    \put(0.90569574,0.22812977){\color[rgb]{1,0,0}\makebox(0,0)[lb]{\smash{$B$}}}%
    \put(0.8713759,0.5050442){\color[rgb]{1,0,0}\makebox(0,0)[lb]{\smash{$A$}}}%
    \put(0.82757044,0.53323679){\color[rgb]{1,0,0}\makebox(0,0)[lb]{\smash{$A$}}}%
  \end{picture}%
\endgroup%